\documentclass[a4paper]{amsart}

\usepackage{amsmath,amssymb,amsfonts}
\usepackage{mathrsfs,latexsym,amsthm,enumerate}
\usepackage{verbatim}
\usepackage{tikz}
\usetikzlibrary{matrix,arrows}

\usepackage[all]{xy}

\newcommand{\superscript}[1]{\ensuremath{^{\textrm{#1}}}}

\let\oldmarginpar\marginpar
\renewcommand\marginpar[1]{\-\oldmarginpar[\raggedleft\footnotesize #1]{\raggedright\footnotesize #1}}

\newcounter{mnotes}
\newcommand{\mnote}[1]{
\superscript{\bf\arabic{mnotes}}
\marginpar{\small \superscript{\arabic{mnotes}}#1}
\addtocounter{mnotes}{1}}

\newcommand{\hidenotes}{\renewcommand{\mnote}[1]{}}

\renewcommand{\2}{\mathbf{2}}
\newcommand{\dom}{{\mathrm{dom}}}

\newcommand{\im}{{\mathrm{im}}}
\newcommand{\D}{\mathcal{D}}

\newcommand{\Sp}{\mathbf{S}}

\newcommand{\Hom}{\mathrm{Hom}}
\newcommand{\ocap}{\wedge}
\newcommand{\ucup}{\vee}

\newcommand{\Up}{\tau^\uparrow} 
\newcommand{\Down}{\tau^\downarrow} 
\newcommand{\Dist}{\mathsf{DL}} 
\newcommand{\DSL}{\mathsf{SDL}} 
\newcommand{\PS}{\mathsf{PS}} 
\newcommand{\LPS}{\mathsf{LPS}}
\renewcommand{\Sp}{\mathcal{S}}
\newcommand{\La}{\mathcal{L}}
\newcommand{\Lat}{\mathsf{Lat}} 
\newcommand{\Skew}{\mathsf{Skew}}
\newcommand{\Skewno}{\mathsf{SkewLR}}

\newcommand{\Etaletwo}{\mathsf{EtalePairs}}

\newcommand{\Sh}{\mathsf{Sh}}
\newcommand{\onto}{\twoheadrightarrow}
\newcommand{\germ}{\mathrm{germ}}
\renewcommand{\L}{\mathcal{L}}
\newcommand{\R}{\mathcal{R}}
\renewcommand{\hat}{\widehat}
\newcommand{\ev}{\mathsf{ev}}

\newcommand{\lincol}{black}
\newcommand{\linth}{thick}
\newcommand{\po}[2][\pocol]{\filldraw[#1](#2) circle (2 pt);}
\newcommand{\li}[1]{\draw[\linth,\lincol] #1;}

\theoremstyle{plain}
\newtheorem{theorem}{Theorem}[section]
\newtheorem{lemma}[theorem]{Lemma}
\newtheorem{proposition}[theorem]{Proposition}
\newtheorem{corollary}[theorem]{Corollary}

\theoremstyle{definition}

\newtheorem{remark}[theorem]{Remark}

\hidenotes

\title{A non-commutative Priestley duality}


\author[A. Bauer]{Andrej Bauer}
\address{A.B.: University of Ljubljana,
Faculty of Mathematics and Physics, \newline
Jadranska 19,
SI-1001, Ljubljana,
Slovenia.}
\email{andrej.bauer\symbol{64}andrej.com}

\author[K. Cvetko-Vah]{Karin Cvetko-Vah}
\address{K.C.-V.: University of Ljubljana,
Faculty of Mathematics and Physics, \newline
Jadranska 19,
SI-1001, Ljubljana,
Slovenia.}
\email{karin.cvetko\symbol{64}fmf.uni-lj.si}

\author[M. Gehrke]{Mai Gehrke}
\address{M.G. (corresponding author): LIAFA, CNRS and Universit\'e Paris Diderot -- Paris 7,\newline
Case 7014,
F-75205 Paris Cedex 13,
+33 1 5727 9404,
France.}
\email{mgehrke\symbol{64}liafa.univ-paris-diderot.fr}

\author[S. J. van Gool]{Samuel J. van Gool}
\address{S.V.G.: LIAFA, Universite Paris Diderot -- Paris 7 and IMAPP, Radboud University Nijmegen, 
P.O. Box 9010
6500 GL Nijmegen\\
The Netherlands.}
\email{samvangool\symbol{64}me.com}

\author[G. Kudryavtseva]{Ganna Kudryavtseva}
\address{G.K.: University of Ljubljana,
Faculty of Computer and Information Science, \newline
Tr\v{z}a\v{s}ka cesta 25,
SI-1001, Ljubljana,
Slovenia.}
\email{ganna.kudryavtseva\symbol{64}fri.uni-lj.si}

\begin{document}

\begin{abstract} 
We prove that the category of left-handed strongly distributive skew lattices with zero and proper homomorphisms is dually equivalent to a category of sheaves over local Priestley spaces. Our result thus provides a non-commutative version of classical Priestley duality for distributive lattices and generalizes the recent development of Stone duality for skew Boolean algebras. 

From the point of view of skew lattices, Leech showed early on that any strongly distributive skew lattice can be embedded in the skew lattice of partial functions on some set with the operations being given by restriction and so-called override. Our duality shows that there is a canonical choice for this embedding. 

Conversely, from the point of view of sheaves over Boolean spaces, our results show that skew lattices correspond to Priestley orders on these spaces and that skew lattice structures are naturally appropriate in any setting involving sheaves over Priestley spaces.
\end{abstract}
\maketitle

\section{Introduction}
Skew lattices \cite{L1,L2} are a non-commutative version of lattices: algebraically, a skew lattice is a structure $(S,\vee,\wedge)$, where $\vee$ and $\wedge$ are binary operations which satisfy the associative and idempotent laws, and certain absorption laws (see \ref{subsec:skewlatdef} below).

Concrete classes of examples of skew lattices occur in many situations. The skew lattices in such classes of examples often have a {\it zero} element, and also satisfy certain additional axioms, which are called {\it strong distributivity} and {\it left-handedness} (see \ref{subsec:distskew} and \ref{subsec:lefthand} below). A (proto)typical class of such examples  is that of {\it skew lattices of partial functions}, which we will describe now. If $X$ and $Y$ are sets, then the collection $S$ of partial functions from $X$ to $Y$ carries a natural skew lattice structure, as follows. If $f, g \in S$ are partial functions, we define $f \wedge g$ to be the {\it restriction} of $f$ by $g$, that is, the function with domain $\dom(f) \cap \dom(g)$, where its value is defined to be equal to the value of $f$. We define $f \vee g$ to be the {\it override} of $f$ with $g$, that is, the function with domain $\dom(f) \cup \dom(g)$, where its value is defined to be equal to the value of $g$ whenever $g$ is defined, and to the value of $f$ otherwise. The {\it zero element} is the unique function with empty domain.

One consequence of the results in this paper is that {\it every left-handed strongly distributive skew lattice with zero can be embedded into a skew lattice of partial functions}. This fact was first proved in \cite[3.7]{L6} as a consequence of the description of the subdirectly irreducible algebras in the variety of strongly distributive skew lattices. Our proof will not depend on this result, and it will moreover provide a canonical choice of an enveloping skew lattice of partial functions. A related result in computer science is described in \cite{BJSV2010}, where the authors give a complete axiomatisation of the structure of partial functions with the operations override and `update', from which the `restriction' given above can also be defined.

In order to state our results precisely, background is needed in skew lattices, Priestley duality, and sheaf theory (see Sections~\ref{s:skew_prel}~and~\ref{s:com}). In particular, we make essential use of the well-known correspondence between \'etal\'e spaces and sheaves. This correspondence allows one to view a sheaf over a space $X$ as a bundle $p:E\to X$ of sets $\{p^{-1}(x)\}_{x\in X}$ such that $p$ is a local homeomorphism (see Subsection~\ref{s:etale_prel} below). The \emph{local sections} of the sheaf are then the partial maps from $X$ to $E$ for which the image of each $x$ in the domain belongs to the stalk $p^{-1}(x)$. This is how sheaves give rise to partial maps. The set of all local sections with clopen domains forms a skew Boolean algebra and if $X$ is also equipped with a partial order, then the local sections with domains that are clopen downsets form a strongly distributive skew lattice. Our duality shows that this accounts for all strongly distributive skew lattices: we will prove that {\it every left-handed strongly distributive skew lattice with zero is isomorphic to a skew lattice of all local sections over clopen downsets of some bundle}. Moreover, it will be a consequence of our duality result that there is a canonical choice for the bundle and base space which represent a given skew lattice. Among all representing bundles, there is an (up to isomorphism) unique bundle $p : E \onto X$ such that $p$ is a local homeomorphism (i.e., {\it \'etale} map) and $X$ is a {\it local Priestley space} (a space whose one-point-compactification is a Priestley space, see Subsection~\ref{subsec:pri} below). This result generalizes both Priestley duality \cite{P1} and recent results on Stone duality \cite{Sto1936} for skew Boolean algebras \cite{BCV,K,K3}, see also \cite{K2}.

Thus, in any setting where sheaves over Priestley spaces are present, in addition to whatever other structure, strongly distributive skew lattice structures are intrinsic. Let us name two examples of settings where sheaves over Priestley spaces are (implicitly) present. First, the classical representation of commutative unital rings as sheaves over their prime ideal spectra with the Zariski topology: Hochster \cite{Hochster} showed that the topological spaces which arise as prime ideal spectra are {\it exactly} the spaces which arise as the Stone duals of distributive lattices, which are now known as {\it spectral} spaces. Much more recently \cite{DuPo}, a sheaf representation over spectral spaces was obtained for MV-algebras, whose category is equivalent to a subcategory of lattice-ordered abelian groups. To place these results precisely in the setting of this paper, it suffices to remark that the category of spectral spaces and spectral functions is {\it isomorphic} to the category of Priestley spaces and continuous monotone functions (cf. \cite{Cor75, Bezetal}). Therefore, any sheaf representation over a spectral space can be equivalently regarded as a sheaf representation over a Priestley space.

In conclusion, our results show that the embeddability of strongly distributive skew lattices in partial function algebras is not coincidental, but a fully structural and natural phenomenon. They also show that strongly distributive skew lattices are intrinsic to sheaves over Priestley spaces and that each such lattice has a canonical embedding into a skew Boolean algebra, namely the skew Boolean algebra of all local sections with clopen domains over the corresponding base. Thus our results open the way to exploring the logic of such structures. In particular, they provide a candidate notion of Booleanization which may in turn allow the development of a non-commutative version of Heyting algebras.


The paper is organized as follows. We first provide background on skew lattices (Section~\ref{s:skew_prel}), Priestley duality (Section~\ref{s:com}), and sheaves (Subsection~\ref{s:etale_prel}). After these preliminaries, we will be ready to state our main theorem (Theorem~\ref{th:main1}), that the categories of left-handed strongly distributive skew lattices and sheaves over local Priestley spaces are dually equivalent. Starting the proof of this theorem, we first give a more formal description of the skew lattice of local sections of an \'etal\'e space, and show that it gives rise to a functor (Section~\ref{s:space_to_lattice}). To show that this functor is part of a dual equivalence, we will describe how to reconstruct the \'etal\'e space from its skew lattice of local sections (Section~\ref{s:new}), and give a general description of this process for an arbitrary left-handed strongly distributive skew lattice (Section~\ref{s:lattice_to_space}). Finally (Section~\ref{s:proof1}), we will put together the results from the preceding sections to prove our main theorem. We close with a few concluding remarks (Section~\ref{s:conc}).

\section{The category $\DSL$ of strongly distributive left-handed skew lattices}\label{s:skew_prel}
For an extensive introduction to the theory of skew lattices we refer the reader to \cite{L1,L2,L6,L3}. To make our exposition self-contained, we collect some definitions and basic facts of the theory.

\subsection{Skew lattices} \label{subsec:skewlatdef} A {\em skew lattice\footnote{In this paper, all skew lattices will be assumed to have a zero element.}} $S$ is an algebra $(S,\wedge,\vee,0)$ of type $(2,2,0)$, such that the operations $\wedge$ and $\vee$ are associative, idempotent and satisfy the absorption identities 
\begin{eqnarray*} 
x\wedge(x\vee y) =  x =  x\vee(x\wedge y), \\ 
(y\vee x)\wedge x  =  x  =  (y\wedge x)\vee x,
\end{eqnarray*}
and the $0$ element satisfies $x\wedge 0 = 0 = 0\wedge x$.
Note that a \emph{lattice} is a skew lattice in which $\wedge$ and $\vee$ are commutative.

The {\em partial order} $\leq$ on a skew lattice $S$ is defined by 
\[ x\leq y \iff x\wedge y= x = y\wedge x, \]
which is equivalent to $x\vee y= y = y\vee x$, by the absorption laws. Note that $0$ is the minimum element in the partial order $\leq$.

If $S$ and $T$ are skew lattices, we say a function $h : S \to T$ is a \emph{homomorphism} if it preserves the operations $\wedge$, $\vee$ and the zero element. We denote by $\Skew_0$ the category of skew lattices with zero  and homomorphisms between them.

\subsection{Lattices form a reflective subcategory of skew lattices}\label{subsec:refl}
If we denote by $\Lat_0$ the category of lattices with zero, then the full inclusion $\Lat_0 \to \Skew_0$ has a left adjoint, which can be explicitly defined using the equivalence relation $\D$, which is well known in semigroup theory \cite{H}. Recall that $\D$ is the equivalence relation on a skew lattice $S$ defined by $x\mathrel{\D} y$ if and only if $x\wedge y\wedge x=x$ and $y\wedge x\wedge y=y$, or equivalently, $x \vee y \vee x = x$ and $y \vee x \vee y = y$. The following is a version of the ``first decomposition theorem for skew lattices''.

\begin{theorem}[\cite{L1}, 1.7] \label{thm:latticereflection} Let $S$ be a skew lattice. The relation $\D$ is a congruence,  and $\alpha_S : S \to S/\D$ is a lattice quotient of $S$. For any homomorphism $h : S \to L$ where $L$ is a lattice, there exists a unique $\bar{h} : S/\D \to L$ such that $\bar{h} \circ \alpha_S = h$.
\end{theorem}
In particular, any skew lattice homomorphism $h : S \to T$ induces a homomorphism between the lattice reflections, which, by a slight abuse of notation, we will also denote by $\overline{h}: S/\D\to T/\D$, and which is defined as the unique lift of the composite $\alpha_T \circ h : S \to T/\D$.

Recall that a lattice homomorphism $k : L_1 \to L_2$ is called \emph{proper} \cite{D} provided that for any $y \in L_2$ there is some $x \in L_1$ such that $k(x) \geq y$. Note that a lattice homomorphism between {\it bounded} lattices is proper if, and only if, it preserves the top element. In the case of skew lattices, we need to consider algebras which may not have a largest element, so we need the `unbounded' version of Priestley duality, where the natural morphisms are the proper homomorphisms, also see Section~\ref{s:com} below. We call a skew lattice homomorphism $h: S \to T$ {\em proper} provided that $\overline{h}$ is proper. 


\subsection{Strongly distributive skew lattices}\label{subsec:distskew}
There are several non-equivalent ways of defining distributivity when passing from lattices to the non-commutative setting. The objects of study in this paper are {\em strongly distributive skew lattices}\footnote{Note that what we call a strongly distributive skew lattice here is termed {\em meet bidistributive and symmetric skew lattice} in \cite{L6}.}, since it turns out that strongly distributive skew lattices allow a generalization of Priestley duality. Here a skew lattice is called \emph{strongly distributive} if it satisfies the identities
\begin{eqnarray}
x\wedge(y\vee z)= (x\wedge y)\vee (x\wedge z), \label{eq:dist1}\\
(y\vee z)\wedge x=(y\wedge x)\vee (z\wedge x). \label{eq:dist2}
\end{eqnarray}

If $S$ is a strongly distributive skew lattice then $S/\D$ is a distributive lattice. In order to state a `converse direction' for this fact, one needs two additional properties which hold in any strongly distributive skew lattice. A skew lattice $S$ is called {\em symmetric} if, for all $x, y \in S$, $x\vee y=y\vee x$ holds if and only if $x\wedge y=y\wedge x$ holds, and {\em normal} if each of the principal subalgebras $x\wedge S\wedge x$ forms a commutative sublattice of $S$. We then have the following result.
\begin{proposition}[\cite{L6},  Theorem 2.5] \label{prop:sdlchar} Let $S$ be a skew lattice. The following are equivalent:
\begin{enumerate}
\item $S$ is a strongly distributive skew lattice;
\item $S$ is normal and symmetric, and the lattice reflection $S/\D$ of $S$ is distributive.
\end{enumerate}
\end{proposition}

\subsection{Left-handed skew lattices}\label{subsec:lefthand}
For our duality, we will focus on strongly distributive skew lattices which are left-handed. A skew lattice $S$ is called {\em left-handed} if it satisfies the identity
\[ x\wedge y\wedge x=x\wedge y \text{, or, equivalently, } x\vee y\vee x=y\vee x.\]
The notion of {\em right-handed} skew lattices is defined dually.

Left-handed strongly distributive skew lattices have some desirable algebraic properties that we collect here, for use in what follows.

\begin{lemma} \label{lem:useful} Let $S$ be a left-handed strongly distributive skew lattice, and let $a, a', b \in S$.
\begin{enumerate}
\item \label{it1} The semigroup $(S,\wedge)$ is left normal, i.e., $b \wedge a \wedge a' = b \wedge a' \wedge a$.
\item \label{it2} If $a, a' \leq b$ and $[a]_\D = [a']_\D$, then $a = a'$.
\end{enumerate}
\end{lemma}
\begin{proof}
(i) By Proposition~\ref{prop:sdlchar}, $S$ is normal. Therefore, using the definition of left-handedness, we get
\[b \wedge a \wedge a' = b \wedge a \wedge a' \wedge b = b \wedge a' \wedge a \wedge b = b \wedge a' \wedge a.\]

(ii) Since $a \, \D \, a'$, left-handedness yields $a \wedge a' = a$ and $a' \wedge a = a'$. Therefore,
\begin{align*}
 a &= b \wedge a &(a \leq b) \\
 &= b \wedge a \wedge a' &(a \, \D \, a') \\
 &= b \wedge a' \wedge a &\text{(item (i))} \\
 &= b \wedge a' &(a \, \D \, a') \\
 &= a' &(a' \leq b). &\qedhere
\end{align*}

\end{proof}

The algebraic object of study in this paper is the category $\DSL$  whose \emph{objects} are left-handed strongly distributive skew lattices with zero, and whose {\em morphisms} are proper homomorphisms. The reason we can restrict to \emph{left-handed} strongly distributive skew lattices without much loss of generality is the following. For a skew lattice $S$, we define the relation $\R$ on $S$ by $x \, \R \, y$ iff $x \wedge y = y$ and $y \wedge x = x$. Dually, we define the relation $\L$ on $S$ by $x \, \L \, y$ iff $x \wedge y = x$ and $y \wedge x = y$. We now have Leech's second decomposition theorem for skew lattices, which says the following.
 
\begin{theorem}[\cite{L1}, Theorem 1.15] \label{th:decomp2} The relations $\L$ and $\R$ are congruences for any skew lattice $S$. Moreover, $S/\L$ is the maximal right-handed image of $S$, $S/\R$ is the maximal left-handed image of $S$, and the following diagram is a pullback:
\begin{center}
\begin{tikzpicture}
\matrix (m) [matrix of math nodes, row sep=3.5em, column sep=2em, text height=1.5ex, text depth=0.25ex] 
{  S & & S/\R \\
      S/\L & & S/\D \\};
\path[->>] (m-1-1) edge (m-2-1);
\path[->>] (m-1-1) edge (m-1-3);
\path[->>] (m-2-1) edge  (m-2-3);
\path[->>] (m-1-3) edge  (m-2-3);
\end{tikzpicture}
\end{center}
\end{theorem}

\subsection{Primitive skew lattices}
In what follows, primitive skew lattices will play an important role. A skew lattice $S$ is called {\em primitive} if it has only one non-zero $\D$-class, or, equivalently, if $S/\D$ is the bounded distributive lattice ${\bf 2}=\{0,1\}$. 

If $T$ is any set, there is a unique (up to isomorphism) primitive left-handed skew lattice $P_T$ with $T$ as its only non-zero $\D$-class (see figure~\ref{fig:primitive}). The operations inside this $\D$-class are determined by lefthandedness: $t \wedge t'=t$ and $t\vee t'=t'$, for any $t,t'\in T$. Note that, clearly, $P_T$ is strongly distributive.

\begin{figure}[htp]
\begin{tikzpicture}
\po{0,0}
\draw node at (0,0) [below] {$0$};
\draw node at (-2,1) [above] {$t$};
\po{-2,1}
\draw node at (-1,1) [above] {$t'$};
\po{-1,1}
\draw node at (0,1) {$...$};
\po{1,1}
\draw node at (1,1) [above] {};
\draw node at (2,1) {$...$};

\li{(0,0)--(-2,1)}
\li{(0,0)--(-1,1)}
\li{(0,0)--(1,1)}
\end{tikzpicture}
\caption{The primitive skew lattice $P_T$\label{fig:primitive}}
\end{figure}
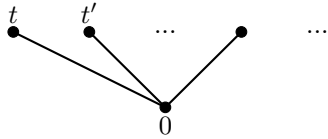

\section{Sheaves over Priestley spaces}\label{s:com}
In this section we first outline a slight modification of classical Priestley duality for bounded distributive lattices which is available for distributive lattices that may not have a largest element. For an extensive introduction to bounded distributive lattices and Priestley duality, we refer the reader to \cite{P,P1,DP}. We then define the category of sheaves over local Priestley spaces, and state our main theorem.

\subsection{The category $\Dist_0$ of distributive lattices with zero}
The {\em objects} of the category $\Dist_0$ are {\em distributive lattices with a zero element}. The \emph{morphisms} of the category $\Dist_0$ are the proper lattice homomorphisms (see \ref{subsec:refl} above).


\subsection{The category $\LPS$ of local Priestley spaces}
Recall that  a {\em Boolean space} \cite{Sto1936} is a compact Hausdorff space in which the clopen sets form a basis.

A subset $E$ of a partially ordered set (poset) $X$ is called an {\em upset} or an {\em upward closed subset} provided that for any $x\in E$ and $y\geq x$ we have $y\in E$. {\em Downsets} or {\em downward closed subsets} of $X$ are defined order-dually. A map $f: X\to Y$ between partially ordered sets is called {\em monotone} if $x_1\leq x_2$ in $X$ implies $f(x_1)\leq f(x_2)$ in $Y$.

We call a triple $(X,\tau,\leq)$ a \emph{partially ordered topological space} if $(X,\tau)$ is a topological space and $(X,\leq)$ is a poset. Let $\Up$ be the set of all open upsets of $X$ and $\Down$ be the set of all open downsets of $X$. It is easy to see that  $\Up$ and $\Down$ are topologies on $X$. A partially ordered topological space $(X,\tau)$ is called {\em totally order-disconnected} \cite{P} provided that for any $x,y\in X$ such that $x\not\leq y$, there exist disjoint clopen sets $U \in \Up$ and $V \in \Down$ such that $x\in U$ and $y\in V$. A \emph{Priestley space} $(X,\tau,\leq)$ is a partially ordered topological space which is compact and totally order-disconnected. The topological reduct $(X,\tau)$ of a Priestley space $(X,\tau,\leq)$ is a Boolean space.  

Priestley duality \cite{P1} is a dual categorical equivalence between the category of bounded distributive lattices $\Dist_{01}$ and the category $\PS$ of Priestley spaces with continuous monotone maps (see below for more details). In order to generalize Priestley duality to the category $\Dist_{0}$, we will use a `local' version of Priestley spaces.

Recall that for any topological space $(X,\tau)$, its {\it one-point-compactification} $(\hat{X}, \hat{\tau})$ is defined by $\hat{X} = X \cup \{\ast\}$, where $\ast \not\in X$, and $U \in \hat{\tau}$ iff either $U \in \tau$, or $\ast \in U$ and $X \setminus U$ is closed and compact for $\tau$. For an ordered space $(X,\tau,\leq)$, we define the \emph{ordered one-point-compactification} $(\hat{X},\hat{\tau},\hat{\leq})$ by letting $(\hat{X},\hat{\tau})$ be the (topological) one-point-compactification, and $\hat{\leq}$ the extension of $\leq$ by adding $\ast$ as a maximum point. 

We say $(X,\tau,\leq)$ is a \emph{local Priestley space} if its ordered one-point-compactification $(\hat{X},\hat{\tau},\hat{\leq})$ is a Priestley space. We define \emph{the category $\LPS$ of local Priestley spaces}, in which a morphism $f : (X,\tau_X,\leq_X) \to (Y,\tau_Y,\leq_Y)$ is the restriction of a continuous monotone map between the one-point-compactifications $f : \hat{X} \to \hat{Y}$ for which $f^{-1}(\ast_Y) = \{\ast_X\}$. 

\begin{remark}
It is possible to give an equivalent definition of this category without referring to the ordered one-point-compactification: local Priestley spaces are exactly the totally order-disconnected spaces for which the space $(X,\Down)$ has a basis consisting of $\tau$-compact open downsets, and $\LPS$-morphisms $(X,\tau_X,\leq_X) \to (Y,\tau_Y,\leq_Y)$ are equivalently described as continuous monotone maps with the further property that the inverse image of a $\tau_Y$-compact set is $\tau_X$-compact.
\end{remark}

\subsection{Priestley duality}\label{subsec:pri} Let $D$ be a bounded distributive lattice. The \emph{spectrum} of $D$ is the Priestley space $\Sp(D) := (X,\tau,\leq)$, defined as follows. The points of $X$ are the prime filters of $D$, $\tau$ is the topology defined by taking as a subbasis for the open sets the collection $\{\hat{a}, \hat{a}^c \ | \ a \in D\}$, where $\hat{a} := \{p \in X \ | \ a \in p\}$, and $\leq$ is the reverse inclusion order on prime filters. A homomorphism $h : D_1 \to D_2$ of bounded distributive lattices yields a continuous monotone function $\Sp(h) : \Sp(D_2) \to \Sp(D_1)$ by sending $p \in \Sp(D_2)$ to $h^{-1}(p) \in \Sp(D_1)$.

Conversely, if $(X,\tau,\leq)$ is a Priestley space, we let $\La(X,\tau,\leq)$ be the bounded distributive lattice of clopen downsets with the set-theoretic operations. A continuous monotone map $f : (X,\tau_X,\leq_X) \to (Y,\tau_Y,\leq_Y)$ yields a lattice homomorphism $\La(f) : \La(Y,\tau_Y,\leq_Y) \to \La(X,\tau_X,\leq_X)$ by sending a clopen downset $U \in \La(Y,\tau_Y,\leq_Y)$ to $f^{-1}(U) \in \La(X,\tau_X,\leq_X)$.

\begin{theorem}[Classical Priestley duality]\label{th:pr_dual} The contravariant functors $\Sp: \Dist_{01}\to \PS$ and $\La:\PS\to\Dist_{01}$ establish a dual equivalence between the categories $\Dist_{01}$ and $\PS$.
The natural isomorphisms $\alpha:1_{\Dist_{01}}\to \La\Sp$ and $\beta: 1_{\PS}\to \Sp\La$ are given by
$$
\alpha_D(a) := \hat{a} =\{p \in \Sp(D): a \in p\},
$$
$$
\beta_X(p) := N_p=\{U \in \La(X,\tau,\leq) : p \in U\}.
$$
\end{theorem}

\begin{remark}\label{rem:natural}
Priestley duality is a {\it natural duality}, in the sense that the functors $\Sp$ and $\La$ are naturally isomorphic to $\hom$-functors into a so-called dualizing object, as follows. Let $\2$ denote both the unique $2$-element distributive lattice in the category $\Dist_{01}$, and the unique $2$-element Priestley space with non-trivial order in the category $\PS$. A prime filter $p$ in a bounded distributive lattice $D$ then corresponds to the lattice homomorphism $h_p : D \to \2$ for which $p = h_p^{-1}(1)$, and a clopen  downset $U$ of a Priestley space $(X,\tau,\leq)$ corresponds to the continuous monotone function $\chi_U : X \to \2$ for which $U = \chi_U^{-1}(0)$.
\end{remark}

Priestley duality can be generalized to a duality between the categories $\Dist_0$ and $\LPS$, as follows. First of all, the category $\Dist_0^h$ of distributive lattices with a zero element and (not necessarily proper) homomorphisms is dually equivalent to the category $\PS_\ast$ of Priestley spaces with a largest element $\ast$, and continuous monotone maps between them. This duality can be described using prime filters, or as a natural duality via the dualizing object $\2$, in a way analogous to Remark~\ref{rem:natural}, cf. \cite[Section 2.8]{DaveyClark}. 

To obtain a duality with the non-full subcategory $\Dist_0$ of $\Dist_0^h$, we now reason as follows. For objects $D, E$ of $\Dist_0$, the proper homomorphisms correspond to those morphisms in $\PS_\ast$ for which $f^{-1}(\ast) = \{\ast\}$. Objects of $\LPS$ can equivalently be described as spaces whose ordered one-point-compactifications lie in $\PS_\ast$; therefore, the category $\LPS$ is isomorphic to the non-full subcategory of $\PS_\ast$ containing only the morphisms which satisfy $f^{-1}(\ast) = \{\ast\}$. Thus, $\LPS$ is dually equivalent to the category $\Dist_0$.

A direct, but slightly more cumbersome, description of this duality can be given as follows. If $D$ is an object in $\Dist_0$, then $\Dist_0^h(D,\2)$ is an object of $\PS_\ast$, where $\ast$ is the constant zero function, which is indeed the largest element, since the order in $\Dist_0^h(D,\2)$ is {\it reverse} pointwise. Let $\Sp(D) := \Dist_0^h(D,\2) \setminus \{\ast\} = \Dist_0(D,\2)$, since the only non-proper homomorphism $D \to \2$ is $\ast$. Then $\Sp(D)$ is a local Priestley space, and the duals of {\it proper} homomorphisms $D \to E$ restrict correctly to functions $\Sp(E) \to \Sp(D)$, by the arguments given in the previous paragraph.
Conversely, if $(X,\tau,\leq)$ is a local Priestley space, let $\La(X,\tau,\leq)$ be the distributive lattice of clopen proper downsets of the one-point-compactification $(\hat{X},\hat{\tau},\hat{\leq})$, or equivalently, compact open downsets of $(X,\tau,\leq)$. We then have the following corollary to Priestley duality:
\begin{corollary} The contravariant functors $\Sp: \Dist_{0}\to \LPS$ and $\La:\LPS\to\Dist_{0}$ establish a dual equivalence between the categories $\Dist_{0}$ and $\LPS$.
\end{corollary}
\begin{remark}\label{rem:natural2}
The duality stated in this corollary is not a natural duality with respect to the dualizing object $\2$ as in Remark~\ref{rem:natural} above. However, it is still true that the set underlying $\Sp(D)$ is in a bijective correspondence with $\Dist_0(D,\2)$, for any $D \in \Dist_0$. Under this correspondence, a basic open $\hat{a}$ gets sent to the set $\{h \in \Dist_0(D,\2) : h(a) = 1\}$.
\end{remark}


\subsection{Sheaves and \'etal\'e spaces}\label{s:etale_prel}

Preliminaries on sheaves and \'{e}tal\'e spaces can be found in any textbook on sheaf theory, e.g. in \cite{Br,MM}. We will recall the basics and notation that we will use here.

Let $X$ be a topological space. We denote by $\Omega(X)$ the poset of open subsets of $X$, ordered by inclusion. In particular, $\Omega(X)$ is a category. A \emph{presheaf} on $X$ is a contravariant functor $E$ from $\Omega(X)$ to the category of sets. In this paper, we will always assume that $E(U) \neq \emptyset$ for all $U \in \Omega(X)$, that is, we only consider presheaves {\it with global support}. If the presheaf $E$ is clear from the context, and $U, V \in \Omega(X)$ with $U \subseteq V$, then we write $(-)|_U : E(V) \to E(U)$ for the morphism $E(U \subseteq V)$, and call it the \emph{restriction map from $V$ to $U$}. 

If $(U_i)_{i \in I}$ is a cover of an open set $U$, then we say a family of elements $(s_i)_{i \in I}$, where $s_i \in E(U_i)$ for each $i \in I$, is \emph{compatible} if for all  $i, j \in I$, $s_i|_{U_i \cap U_j} = s_j|_{U_i \cap U_j}$.
A presheaf $E$ on $X$ is called a \emph{sheaf} if for any such compatible family there exists a unique $s \in E(U)$ such that $s|_{U_i} = s_i$ for all $i \in I$. For reasons that will become apparent later, we will also denote this unique element $s$ by $\bigvee_{i \in I} s_i$, and call it the \emph{patch} of the family $(s_i)_{i \in I}$.

If $E$ is a sheaf on a topological space $X$ and $f : X \to Y$ is a continuous map, we define the functor $f_\ast E$ on $\Omega(Y)$ on objects by $(f_\ast E)(V) := E(f^{-1}(V))$, and we call $f_\ast E$ the \emph{direct image sheaf} of $E$ under $f$. It is a well known fact in sheaf theory that $f_\ast E$ is indeed again a sheaf \cite[Ch. II, \S 1]{MM}.


In this paper, a \emph{morphism} from a sheaf $E$ on $X$ to a sheaf $F$ on $Y$ is a pair $(f,\lambda)$, where $f : X \to Y$ is a morphism of the base spaces, and $\lambda : F \Rightarrow f_\ast E$ is a natural transformation. In the proof of Proposition~\ref{prop:fullfaithful}, we will use the following lemma.
\begin{lemma}\label{lem:basicdet}
Suppose $(f,\lambda)$ and $(f,\lambda')$ are morphisms from a sheaf $E$ on $X$ to a sheaf $F$ on $Y$, and suppose that $\mathcal{B}$ is a basis for the space $Y$. If, for all $V \in \mathcal{B}$, $\lambda_V = \lambda'_V$, then $\lambda = \lambda'$.
\end{lemma}

We now sketch the basic correspondence between sheaves and \'etal\'e spaces. See \cite[Ch. II, \S 5]{MM} for more details.

Let $X$ be a topological space. A {\it bundle over $X$} is a topological space $E$ together with a continuous map $p : E \to X$. An {\em \'{e}tal\'e space} or {\em \'etal\'e bundle over $X$} is a bundle $p : E \to X$ which is a {\it local homeomorphism}, that is, for any $e \in E$, there exists an open neighbourhood $V$ of $e$ such that $p(V)$ is open in $X$ and $p|_V : V \to p(V)$ is a homeomorphism. 
If $U$ is an open subset of $X$, a \emph{(local) section over U} is a continuous map $s : U \to E$ such that $p \circ s = \mathrm{id}_U$. We denote by $E(U)$ the set of sections over $U$. 
The equivalence classes induced by $p$ are called \emph{stalks} or \emph{fibers}: for $x \in X$, we denote the stalk $p^{-1}(\{x\})$ by $E_x$. 

If $p : E \to X$ is an \'etal\'e space, then the assignment $U \mapsto E(U)$, the local sections over $U$, naturally extends to a sheaf on $X$: if $U \subseteq V$, then we have the map $E(V) \to E(U)$ which sends a local section $s$ over $V$ to its restriction $s|_U$ over $U$. We call $E$ the {\it sheaf associated to the \'etale map $p$}.

If $F$ is a sheaf on $X$, then for any $x \in X$ we define the \emph{stalk} $F_x$ to be the colimit of the diagram of sets $F(U)$, where $U$ ranges over the open neighbourhoods of $x$. More explicitly, 
\[F_x = \left(\bigsqcup_{x \in U} F(U)\right)/\sim_x,\]
where, for $s \in F(U)$ and $t \in F(V)$, we have $s \sim_x t$ iff there exists an open neighbourhood $W$ of $x$ such that $W \subseteq U \cap V$ and $s|_W = t|_W$. The classes in $F_x$ are called \emph{germs} and denoted by $\germ_x s$. The {\it \'etal\'e space associated to $F$} has $\bigsqcup_{x \in X} F_x$ as its underlying set. Any $s \in F(U)$ yields a function $\hat{s} : U \to \bigsqcup_{x \in X} F_x$ by sending $x \in U$ to $\germ_x s$. The topology on $\bigsqcup_{x \in X} F_x$ is defined by taking the sets $\hat{s}(U)$ as a basis, where $U$ ranges over $\Omega(X)$ and $s$ ranges over $F(U)$. One may now prove that these assignments from an \'etale map over $X$ to a sheaf over $X$ and vice versa are well-defined and mutually inverse up to isomorphism, as in \cite[Corollary II.6.3]{MM}. Sheaves with global support correspond to \'etale maps which are surjective.

\subsection{The category of sheaves over local Priestley spaces}\label{subsec:sheaves}
If $E$ is a sheaf on a topological space $X$ and $f : X \to Y$ is a continuous map, recall that the direct image sheaf $f_\ast E$ on $\Omega(Y)$ is defined on objects by $(f_\ast E)(V) := E(f^{-1}(V))$.

We will denote by $\Sh(\LPS)$ the category of \emph{sheaves over local Priestley spaces}: an \emph{object} is $(X,\tau,\leq,E)$, where $(X,\tau,\leq)$ is a local Priestley space, and $E$ is a sheaf.

A \emph{morphism} from $(X,\tau,\leq,E)$ to $(Y,\tau,\leq,F)$ is a pair $(f,\lambda)$, where the function $f : (X,\tau,\leq) \to (Y,\tau,\leq)$ is a morphism in $\LPS$, and $\lambda : F \Rightarrow f_\ast E$ is a natural transformation; see the diagram in Figure~\ref{fig:sheafmorphism}. If $(f,\lambda) : (X,E) \to (Y,F)$ and $(g,\mu) : (Y,F) \to (Z,G)$ are morphisms in $\Sh(\LPS)$, their composition is defined by $(gf, \sigma)$, where $\sigma_U := \lambda_{g^{-1}(U)} \circ \mu_U$.
\begin{figure}[htp]
\begin{tikzpicture}
\matrix (m) [matrix of math nodes, row sep=3.5em, column sep=2em, text height=1.5ex, text depth=0.25ex] 
{ E & f_\ast E & & & F \\
      & X & & Y \\};
\path[->] (m-2-2) edge node[above] {$f$} (m-2-4);
\path[->] (m-1-5) edge node[above] {$\lambda$} (m-1-2);
\path[dashed,->] (m-1-1) edge (m-2-2);
\path[dashed,->] (m-1-2) edge (m-2-4);
\path[dashed,->] (m-1-5) edge (m-2-4);
\end{tikzpicture}
\caption{A morphism in the category $\Sh(\LPS)$. \label{fig:sheafmorphism}}
\end{figure}
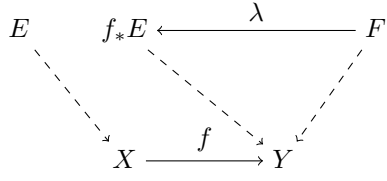

\subsection{Statement of the main theorem}
We are now ready to state our main theorem.
\begin{theorem}\label{th:main1}
The category $\DSL$ of left-handed strongly distributive skew lattices is dually equivalent to the category $\Sh(\LPS)$ of sheaves over local Priestley spaces.
\end{theorem}

The proof of this theorem will take up the rest of this paper.
 In Section~\ref{s:space_to_lattice}, we will define a left-handed strongly distributive skew lattice from a sheaf over a local Priestley space, and extend this assignment to a functor. We will then show how to retrieve the original sheaf from this in Section~\ref{s:new}. This will lead to the right way to associate a sheaf over a local Priestley space to a left-handed strongly distributive skew lattice in Section~\ref{s:lattice_to_space}. In Section~\ref{s:proof1} we will put all of this together to prove Theorem~\ref{th:main1}.

\section{From an \'{e}tal\'e space to a skew lattice}\label{s:space_to_lattice}

In this section, let $X=(X,\tau,\leq)$ be a local Priestley space and let $p : E \onto X$ be an \'etal\'e space over $X$. We denote the corresponding sheaf of local sections by $E$ as well. From these data, we will now construct a left-handed strongly distributive skew lattice $S$.

Let us denote by $L := \La(X)$ the distributive lattice of compact open downsets of $X$. We define the underlying set of $S$ to be $\bigsqcup_{U \in L} E(U)$, that is, the set of all local sections over all compact open downsets of $X$. We now define operations $\ucup$ and $\ocap$ on $S$ that will make it into a left-handed strongly distributive skew lattice.

Let $U, V \in L$ and $a\in E(U)$, $b\in E(V)$. We define the {\em override} $a\ucup b$ to be the local section
over $U \cup V$ given by
\begin{equation}\label{eq:rqu} (a\ucup b)(x) := \left\lbrace\begin{array}{ll}b(x), & \text{ if } x\in V,\\
a(x), & \text{ if } x\in U\setminus V.
\end{array}\right.
\end{equation}
Note that this indeed defines a continuous map $U \cup V \to E$, so $a \ucup b \in E(U \cup V)$. Viewing $E$ as a sheaf over $X$, note that $a \ucup b$ is the patch of the compatible family consisting of the two elements $a|_{U \setminus V}$ and $b|_V$, that is, $a \ucup b = a|_{U \setminus V} \vee b|_V$.

We define the {\em restriction} $a\ocap b$ to be the section in $E(U\cap V)$ given by
\begin{equation}\label{eq:lqi} (a\ocap b)(x) := a(x) \text{ for all } x\in U\cap V.
\end{equation}
Viewing $E$ as a sheaf, $a \ocap b$ is simply the restriction $a|_{U \cap V}$.

Finally, we let the \emph{zero element}, $0$, be the unique element of $E(\emptyset)$. Indeed, since $E$ is a sheaf and any element of $E(\emptyset)$ is a patch of the empty compatible family, $E(\emptyset)$ contains exactly one element.

In the following proposition, we collect some basic properties of the algebra $S$ that we constructed here.
\begin{proposition}\label{prop:basicprops}
Let $p : E \onto X$ be an \'etal\'e space over a local Priestley space, $L := \La(X)$ and $(S, \ocap, \ucup, 0)$ the algebra on $S = \bigsqcup_{U \in L} E(U)$ defined in \eqref{eq:rqu} and \eqref{eq:lqi}. Then the following hold.
\begin{enumerate}
\item The algebra $S$ is a left-handed strongly distributive skew lattice. 
\item The lattice reflection $S/\D$ of $S$ is isomorphic to $L$.
\item The order on $S$ is given by $a \leq b$ if and only if $a$ is a restriction of $b$.
\end{enumerate}
\end{proposition}

\begin{proof} (i) It is known \cite{L6} and easy to check that the skew lattice ${\mathcal P}(X,E)$ of all partial maps from $X$ to $E$ is a left-handed strongly distributive skew lattice.  It is easy to verify that $S$ is a subalgebra of ${\mathcal P}(X,E)$, and therefore it is also a left-handed strongly distributive skew lattice.

(ii) Note that the relation $\D$ on $S$ is given by $a \, \D \, b$ if and only if $\dom(a) = \dom(b)$ (recall that the notation $\dom(a)$ denotes the domain of the function $a$). Hence, $S/\D$ is indeed isomorphic to the lattice of domains, $L$. 

(iii) By definition of $\leq$, we have $a\leq b$ if and only if $a\ocap b=a = b\ocap a$. The statement now follows from the definition of $\ocap$.\qedhere

\end{proof}

Let us call the left-handed strongly distributive skew lattice $S$ the {\em dual algebra} of the \'{e}tal\'e space $p : E \onto X$. We will sometimes denote $S$ by $E^{\star}$ or $(E,p,X)^{\star}$, to emphasize that it is constructed from the \'etal\'e space $(E,p,X)$. We now use the above construction to define a contravariant functor from $\Sh(\LPS)$ to $\DSL$, which will be one of the equivalence functors of the duality in Theorem~\ref{th:main1}.

Let $E$ and $F$ be sheaves over local Priestley spaces $(X,\tau,\leq)$ and $(Y,\tau,\leq)$, respectively. The naturally associated \'etal\'e spaces $E \onto X$ and $F \onto Y$ yield dual algebras $E^\star$ and $F^\star$. Suppose $(f,\lambda)$ is a morphism from $E$ to $F$, as in Figure~\ref{fig:sheafmorphism} in \ref{subsec:sheaves}. We will define a skew lattice morphism $(f,\lambda)^\star : F^\star \to E^\star$.

Let $a \in F^\star$, so $a \in F(U)$ for some $U \in \La(Y)$. By classical Priestley duality, we have $f^{-1}(U) \in \La(X)$. We now define 
\[(f,\lambda)^\star(a) := \lambda_U(a),\]
which is an element of $f_\ast E(U) = E(f^{-1}(U)) \subseteq E^\star$.

\begin{lemma}\label{lem:dslmorph}
The function $(f,\lambda)^\star : F^\star \to E^\star$ is a morphism in $\DSL$ for which $\overline{(f,\lambda)^\star} = \La(f)$.
\end{lemma}
\begin{proof}
Let us write $h$ for the function $(f,\lambda)^\star$. We show in detail that $h$ preserves the operation $\ocap$, and leave it to the reader to verify that $h$ preserves $\ucup$ and $0$, since the proofs are similar. Let $a \in F(U)$, $b \in F(V)$. By definition of $\ocap$, we have $h(a) \ocap h(b) = h(a)|_{f^{-1}(U) \cap f^{-1}(V)}$. By naturality of $\lambda$, the following diagram commutes:
\begin{center}
\begin{tikzpicture}
\matrix (m) [matrix of math nodes, row sep=3.5em, column sep=2em, text height=1.5ex, text depth=0.25ex] 
{ F(U) & & f_\ast E(U) \\
F(U \cap V) & & f_\ast E(U \cap V) \\};
\path[->] (m-1-1) edge node[above] {$\lambda_U$} (m-1-3);
\path[->] (m-2-1) edge node[below] {$\lambda_{U \cap V}$} (m-2-3);

\path[->] (m-1-1) edge node[left] {$(-)|_{U \cap V}$} (m-2-1);
\path[->] (m-1-3) edge node[right] {$(-)|_{f^{-1}(U \cap V)}$} (m-2-3);
\end{tikzpicture}
\end{center}
In particular, we get 
\begin{align*}
h(a \ocap b) = \lambda_{U \cap V}(a \ocap b) = \lambda_{U \cap V}(a|_{U \cap V}) &= \lambda_U(a)|_{f^{-1}(U \cap V)} \\
&= h(a)|_{f^{-1}(U) \cap f^{-1}(V)} = h(a) \ocap h(b).
\end{align*}

Further note that $\bar{h} : F^\star\!/\D \to E^\star\!/\D$ is exactly the proper homomorphism $f^{-1} = \La(f)$ dual to $f$ in classical Priestley duality. Hence, $h$ is a morphism in $\DSL$ and $\bar{h} = \La(f)$.
\end{proof}
In conclusion, we can record the following proposition.
\begin{proposition}
The assignments $(E,p,X) \mapsto (E,p,X)^\star$ and $(f,\tau) \mapsto (f,\tau)^\star$ define a contravariant functor $(-)^\star$ from $\Sh(\LPS)$ to $\DSL$.
\end{proposition}
\begin{proof}
By Proposition~\ref{prop:basicprops}(i) and Lemma~\ref{lem:dslmorph}, the assignments are well-defined. We leave functoriality to the reader.
\end{proof}

\section{Reconstructing an \'etal\'e space from its dual algebra}\label{s:new}
In this section, we show how a sheaf $E$ over a local Priestley space $X$ can be reconstructed (up to homeomorphism) from its dual algebra $E^\star$, defined in the previous section. This will be the main motivation for the construction leading to the definition of a contravariant functor $(-)_\star : \DSL \to \Sh(\LPS)$ in the next section.

In the remainder of this section, let $E$ be a sheaf over a local Priestley space $X$, and let $p : E \onto X$ be the \'etal\'e space associated to the sheaf. Let $E^\star$ be the dual algebra of $E$, $L := E^{\star}\!/\D$ its lattice reflection and $\alpha: E^\star \to L$ the natural quotient map.

\subsection{Reconstructing the base space}\label{subsec:basespace}
We first note that we can reconstruct the base space $X$ from the left-handed strongly distributive skew lattice $E^\star$. By Proposition~\ref{prop:basicprops}(i), $L$ is isomorphic to $\La(X)$. Hence, $X$ is homeomorphic to the space $\Sp(L)$, by classical Priestley duality.
A point of $\Sp(L)$ can be concretely given by a morphism $L \to \2$ in $\Dist_0$, by Remark~\ref{rem:natural2}. By Theorem~\ref{thm:latticereflection}, the hom-set $\Dist_0(L,\2) = \Dist_0(E^\star\!/\D,\2)$ is naturally isomorphic to the hom-set $\DSL(E^\star,\2)$, because $\2$ is a lattice. In summary, we obtain
\begin{equation}\label{eq:homeo}
X \cong \Sp(L) \cong \Dist_0(L,\2) \cong \DSL(E^\star,\2),
\end{equation}
where the topology on $\DSL(E^\star,\2)$ is given by taking as a basis the sets of the form $\{h \in \DSL(E^\star,\2) : h(a) = 1\}$ and their complements, where $a$ ranges over $E^\star$.

\subsection{Reconstructing the stalks}
We will now reconstruct, for any $x \in X$, the stalk $E_x$ above it. Fix $x \in X$. Let $P_x$ be the primitive skew lattice whose non-zero $\D$-class is the set $E_x$. Then we have a natural evaluation homomorphism $\ev_x : E^\star \to P_x$, defined by
\[ \ev_x(a) := \left\{ \begin{array}{cc} a(x) &\text{ if $x \in \dom(a)$} \\
							0   &\text{ otherwise}.
				\end{array}\right.
				\]
Note that the composition $\alpha \circ \ev_x : E^\star \to \2$ is exactly the map $h_x$ naturally associated to $x$ in (\ref{eq:homeo}): it sends $a \in E^\star$ to $1$ iff $x \in \dom(a)$. We can now characterize the kernel of the homomorphism $\ev_x$ by an algebraic property which only refers to $\wedge$, $\vee$, $0$ and the map $h_x$, as follows.

\begin{lemma}\label{lem:simhxchar}
Let $x \in X$. For any $a, b \in E^\star$, the following are equivalent:
\begin{enumerate}
\item $\ev_x(a) = \ev_x(b)$;
\item there exist $c, d \in E^\star$ such that $h_x(c) = 0$, $h_x(d) = 1$, and $(a \wedge d) \vee c = (b \wedge d) \vee c$.
\end{enumerate}
\end{lemma}
\begin{proof}
(i) $\Rightarrow$ (ii). If $\ev_x(a) = 0$, then, since $h_x$ is proper, we can pick some $d$ with $h_x(d) = 1$. Put $c := a \vee b$. Note that $h_x(c) = 0$ since $x \not\in \dom(c) = \dom(a) \cup \dom(b)$. Now $(a \wedge d) \vee c$ and $(b \wedge d) \vee c$ are both equal to $a \vee b$, as required.

If $\ev_x(a) \neq 0$, then $x \in \dom(a) \cap \dom(b)$. Since $a$ and $b$ are continuous sections, their equalizer $\| a = b \| = \{x \in \dom(a) \cap \dom(b) \ | \ a(x) = b(x)\}$ is open in $X$, and it contains $x$, so there exist compact open downsets $U$ and $V$ of $X$ such that $x \in U \cap V^c \subseteq \| a = b \|$. Pick some $d \in E(U)$ and $c \in E(V)$. Then $x \not\in \dom(c) = V$, so $h_x(c) = 0$, and $x \in \dom(d) = U$, so $h_x(d) = 1$. It is clear from the definitions of $\wedge$ and $\vee$ that $\dom((a \wedge d) \vee c) = U \cup V = \dom((b \wedge d) \vee c)$, and that the values of $(a \wedge d) \vee c$ and $(b \wedge d) \vee c$ are equal, since $a$ and $b$ are equal on $U \cap V^c$ by construction.

(ii) $\Rightarrow$ (i). Note that (ii) implies $h_x(a) = h_x(b)$, since $h_x$ is a homomorphism. Hence, we have either $x \not\in \dom(a)$ and $x \not\in \dom(b)$, or $x \in \dom(a)$ and $x \in \dom(b)$. In the first case, (i) clearly holds and we are done. If $x \in \dom(a) \cap \dom(b)$, we show that $a(x) = b(x)$. Pick $c, d \in E^\star$ such that $h_x(c) = 0$, $h_x(d) = 1$ and $(a \wedge d) \vee c = (b \wedge d) \vee c$. Since $x \not\in \dom(c)$ and $x \in \dom(d)$, we get from the definitions of $\wedge$ and $\vee$ that $((a \wedge d) \vee c)(x) = a(x)$ and $((b \wedge d) \vee c)(x) = b(x)$, so $a(x) = b(x)$, as required.
\end{proof}
Hence, given a point $x \in X$, we define a relation $\sim_x$ on $E^\star$ by
\begin{eqnarray*}
a \sim_{x} b \iff \exists c, d \in S : h_x(c) = 0, h_x(d) = 1, \text{ and }(a \wedge d) \vee c = (b \wedge d) \vee c,
\end{eqnarray*}
and we immediately obtain:
\begin{proposition}\label{prop:stalkfromquotient}
Let $x \in X$. The relation $\sim_{x}$ is a skew lattice congruence on $E^\star$, and there is an isomorphism between $E^\star\!/\!\!\sim_{x}$ and $P_x$, which takes the quotient map $E^\star \onto E^\star\!/\!\!\sim_x$ to the evaluation map $\ev_x : E^\star \onto P_x$.
\end{proposition}
\begin{proof}
The preceding lemma exactly shows that $\sim_{x}$ is the kernel of the morphism $\ev_x$. The result now follows from the first isomorphism theorem of universal algebra.
\end{proof}

\subsection{Reconstructing the \'etal\'e space}
For a primitive skew lattice $P$, we denote by $P^1$ the unique non-zero $\D$-class of $P$, considered as a set.
\begin{corollary}
The \'etal\'e space $p : E \onto X$ is isomorphic to $q : (E^\star)_\star \onto X$, where
\begin{itemize}
\item the set underlying the space $(E^\star)_\star$ is 
\[\bigsqcup_{x \in X} (E^\star\!/\!\!\sim_{x})^1 = \{(x,[a]_{\sim_x}) : x \in X, [a]_{\sim_x} \in (E^\star\!/\!\!\sim_x)^1\},\]
\item the function $q : (E^\star)_\star \onto X$ sends an element of the disjoint union to its index $x \in X$,
\item the topology on $(E^\star)_\star$ is given by taking as a basis of open sets the sets of the form
\[ \hat{a} := \{(x, [a]_{\sim_x}) \ | \ x \in \dom(a)\},\]
where $a$ ranges over the elements of $E^\star$.
\end{itemize}
\end{corollary}
\begin{proof}
Define a map $\psi : E \to (E^\star)_\star$ by sending $e \in E_x$ to $(x, [a]_{\sim_{x}})$, where $a$ is any local section for which $a(x) = e$; such a section exists because $p$ is an \'etale map, and the value of $\psi(e)$ does not depend on the choice of $a$ because of Lemma~\ref{lem:simhxchar}. By Proposition~\ref{prop:stalkfromquotient}, $\psi$ is a bijection. It is not hard to see from the definition of the topologies on $E$ and $(E^\star)_\star$ that $\psi$ is open and continuous. Hence, $\psi$ is a homeomorphism, which clearly commutes with the \'etale maps.
\end{proof}

\section{From a left-handed strongly distributive skew lattice to an \'etal\'e space}\label{s:lattice_to_space}
In this section, we generalize the construction from the previous section to an \emph{arbitrary} left-handed strongly distributive skew lattice $S$. This is the main contribution of this paper, and it is the key to the proof that the functor $(-)^\star$ defined in Section~\ref{s:space_to_lattice} is part of a contravariant equivalence of categories.

Let $S$ be a left-handed strongly distributive skew lattice. We will define an \'etal\'e space $q : S_\star \to X$ over a local Priestley space. 

\subsection{The base space $X$}
Recall from Proposition~\ref{prop:sdlchar} that $S/\D$ is a distributive lattice with $0$. By Remark~\ref{rem:natural2} and Theorem~\ref{thm:latticereflection}, the set underlying the local Priestley space $\Sp(S/\D)$ is in a bijective correspondence with the set $\DSL(S,\2)$. A topology on $\DSL(S,\2)$ is given by taking as a basis the sets of the form $\hat{a} = \{h : S \to \2 \ | \ h(a) = 1\}$ and their complements, where $a$ ranges over $S$. With this topology, $\DSL(S,\2)$ is homeomorphic to the local Priestley space $\Sp(S/\D)$. We will denote this space by $X$, and we will define an \'etal\'e space over $X$.

\subsection{A maximal primitive quotient}
Inspired by the results in the previous section, for $h \in X = \DSL(S,\2)$, we define the relation $\sim_h$ as follows:
\begin{eqnarray}\label{eq:defsimh}
a \sim_h b \iff &\exists c, d \in S : h(c) = 0, h(d) = 1, \text{ and } (a \wedge d) \vee c = (b \wedge d) \vee c.
\end{eqnarray}

The following proposition is now the central technical result that we need to construct $S_\star$.
\begin{proposition}
Let $S$ be a left-handed strongly distributive skew lattice, and $h \in \DSL(S,\2)$. The following properties hold:
\begin{enumerate}
\item The relation $\sim_h$ is a skew lattice congruence on $S$ which refines $\ker(h)$.
\item The quotient skew lattice $S/\!\!\sim_h$ is primitive and the diagram
\begin{center}
\begin{tikzpicture}
\matrix (m) [matrix of math nodes, row sep=3.5em, column sep=2em, text height=1.5ex, text depth=0.25ex] 
{  S & & \\
      S/\!\!\sim_h & & \2 \\};
\path[->>] (m-1-1) edge node[left] {$\pi$} (m-2-1);
\path[->] (m-2-1) edge node[above] {$\alpha$} (m-2-3);
\path[->] (m-1-1) edge node[above] {$h$} (m-2-3);
\end{tikzpicture}
\end{center}
commutes.
\item For any commuting diagram in $\DSL$ of the form
\begin{center}
\begin{tikzpicture}
\matrix (m) [matrix of math nodes, row sep=3.5em, column sep=2em, text height=1.5ex, text depth=0.25ex] 
{  S & & \\
      P' & & \2 \\};
\path[->>] (m-1-1) edge node[left] {$\pi'$} (m-2-1);
\path[->] (m-2-1) edge node[above] {$\alpha$} (m-2-3);
\path[->] (m-1-1) edge node[above] {$h$} (m-2-3);
\end{tikzpicture}
\end{center}
where $P'$ is primitive, there is a unique factorization $t : S/\!\!\sim_h \to P'$ such that $t \circ \pi = \pi'$.
\end{enumerate}
\end{proposition}

\begin{proof}
(i) It is clear that $\sim_h$ is reflexive and symmetric. For transitivity, if $a \sim_h f \sim_h b$, pick $c, c', d, d' \in S$ are such that $h(c) = 0 = h(c')$, $h(d) = 1 = h(d')$, $(a \wedge d) \vee c = (f \wedge d) \vee c$, and $(b \wedge d') \vee c' = (f \wedge d') \vee c'$. Put $c'' := c \vee c'$ and $d'' := d \wedge d'$, then $h(c'') = 0$ and $h(d'') = 1$ since $h$ is a homomorphism. One may now check that the elements $(a \wedge d'') \vee c''$ and $(b \wedge d'') \vee c''$ are in the same $\D$-class, and that both are below $f \vee c''$.
%
Therefore, by Lemma~\ref{lem:useful}(ii), $(a \wedge d'') \vee c'' = (b \wedge d'') \vee c''$, and we obtain $a \sim_h b$.

Suppose $a \sim_h a'$, and let $b \in S$. We first show that $a \vee b \sim_h a' \vee b$ and $b \vee a \sim_h b \vee a'$. Pick $c, d \in S$ such that $h(c) = 0$, $h(d) = 1$ and $(a \wedge d) \vee c = (a' \wedge d) \vee c$. We use distributivity and left-handedness to show that $((a \vee b) \wedge d) \vee c = ((a' \vee b) \wedge d) \vee c$:
\begin{align*}
((a \vee b) \wedge d) \vee c &= (a \wedge d) \vee (b \wedge d) \vee c &\text{(strong distributivity)} \\
&= (a \wedge d) \vee c \vee (b \wedge d) \vee c &\text{(left-handedness)}\\
&= (a' \wedge d) \vee c \vee (b \wedge d) \vee c &\text{(assumption)} \\
&= (a' \wedge d) \vee (b \wedge d) \vee c &\text{(left-handedness)} \\
&= ((a' \vee b) \wedge d) \vee c. &\text{(strong distributivity)}
\end{align*}
The proof that $((b \vee a) \wedge d) = ((b \vee a') \wedge d)$ is similar, but slightly simpler.

The proof that $\sim_h$ is also a congruence for the operation $\wedge$ on both sides proceeds along similar lines, using left normality (Lemma~\ref{lem:useful}(i)), and is left for the reader to check.

To see that $\sim_h \; \subseteq \ker(h)$, suppose $a \sim_h b$ and pick $c, d \in S$ as in the definition of $\sim_h$. Then
\[ h(a) = (h(a) \wedge h(d)) \vee h(c) = h((a \wedge d) \vee c) = h((b \wedge d) \vee c)) = h(b).\]

(ii) We will show that the $\D$-classes of the skew lattice $S/\!\!\sim_h$ are exactly $h^{-1}(0)$ and $h^{-1}(1)$, which is clearly enough for the proof of this item. Since $h$ is proper, fix $a \in S$ such that $h(a) = 1$. We first claim that the $\D$-class of $[0]_{\sim_h}$ is $h^{-1}(0)$. If $b \sim_h 0$ then $h(b) = h(0) = 0$. Conversely, if $h(b) = 0$, one may prove that (\ref{eq:defsimh}) holds by taking $c := b$ and $d := a$, concluding the proof of the claim. We will now show that the $\D$-class of $[a]_{\sim_h}$ is $h^{-1}(1)$. Suppose $b \in S$ is such that $h(b) = 1$. We claim that $[a]_{\sim_h} \D [b]_{\sim_h}$. By definition of $\D$, we need to show that $[a \wedge b]_{\sim_h} = [a]_{\sim_h}$ and $[b \wedge a]_{\sim_h} = [b]_{\sim_h}$. Both  of these equalities hold indeed, because we can take $c := 0$ and $d := a \wedge b$ to prove that (\ref{eq:defsimh}) holds.

(iii) Suppose that $\pi' : S \onto P'$ is a primitive quotient of $S$ such that $\alpha \circ \pi' = h$. If $t : S/\!\!\sim_h \to P'$ is a factorization such that $t \circ \pi = \pi'$, then for any $a \in S$ we must have $t([a]_{\sim_h}) = \pi'(a)$, proving that $t$ is unique if it exists.

We now show that the assignment $[a]_{\sim_h} \mapsto \pi'(a)$ does not depend on the choice of representative for the class $[a]_{\sim_h}$. Suppose $a \sim_h a'$. If $h(a) = 0 = h(a')$, then $[\pi'(a)]_\D = h(a) = 0$ so $\pi'(a) = 0$ since the $\D$-class of $0$ only contains $0$ itself, and similarly $\pi'(a') = 0$. Otherwise, we have $h(a) = 1 = h(a')$. Pick $c, d \in S$ such that $h(c) = 0$, $h(d) = 1$ and $(a \wedge d) \vee c = (a' \wedge d) \vee c$.  As before, since $h(c) = 0$, we have $\pi'(c) = 0$. Since $P'$ is primitive, we have, for any non-zero $x, y \in P'$, that $x \wedge y = x$. Hence
\[\pi'(a) = \pi'(a) \wedge \pi'(d) = (\pi'(a) \wedge \pi'(d)) \vee \pi'(c) = \pi'((a \wedge d) \vee c),\]
and similarly $\pi'(a') = \pi'((a' \wedge d) \vee c)$. So $\pi'(a) = \pi'(a')$, since $(a \wedge d) \vee c = (a' \wedge d) \vee c$.
\end{proof}
\begin{remark}
In the light of this proposition, more can be said about the structure of primitive quotients of a left-handed strongly distributive skew lattice $S$. We may put a partial order on quotients of $S$ by saying a quotient $q : S \to Q$ is {\it below} another quotient $q' : S \to Q'$ if the map $q$ factors through $q'$. Suppose $p : S \to P$ is any primitive quotient of $S$. Then $h := \alpha \circ p : S \to \2$ is a minimal quotient of $S$ below the primitive quotient $P$, and $S/\!\!\sim_h$ is a maximal primitive quotient of $S$ which is above $P$. The partially ordered set of primitive quotients of $S$ is thus partitioned, and each primitive quotient lies between a unique maximal and minimal primitive quotient of $S$. The minimal primitive quotients of $S$ are exactly the elements of the base space $X$, and the non-zero elements of the maximal primitive quotients will be the elements of the \'etal\'e space $S_\ast$, see below.
\end{remark}

\begin{remark}\label{rem:primefilters}
An alternative way to define the equivalence relation $\sim_h$ on $S$ is the following. Let us call a subset $F$ of $S$ a \emph{preprime filter over $h$} if it satisfies the following properties:
\begin{enumerate}
\item if $a \in F$, $b \in S$ and $a \leq b$, then $b \in F$;
\item if $a, b \in F$ then $a \wedge b \in F$;
\item if $a \in F$, $b \in S$ and $h(b) = 0$, then $a \vee b \in F$;
\item if $a \in F$, then $h(a) = 1$;
\item if $b \in S$ and $h(b) = 1$, then there is $a \in F$ such that $[a]_\D = [b]_\D$.
\end{enumerate}
We call a preprime filter over $h$ a \emph{prime filter over $h$} if it is minimal among the preprime filters over $h$. One may then show that the non-zero equivalence classes in $S/\!\!\sim_h$ (viewed as subsets of $S$) are exactly the prime filters over $h$. Therefore, the equivalence relation $\sim_h$ can also be described as the equivalence relation inducing the partition whose classes are the prime filters over $h$, and $h^{-1}(0)$.
\end{remark}

\subsection{The \'etal\'e space}
We are now ready to define the \'etal\'e space $S_\star$. The stalk over $h \in X$ will be the non-zero $\D$-class of $S/\!\!\sim_h$, or, equivalently, the set of prime filters over $h$, as defined in Remark~\ref{rem:primefilters}. Put more formally, the underlying set of the \'etal\'e space $S_\star$ is 
\[S_\star := \bigsqcup_{h \in X} (S/\!\! \sim_h)^1 = \{(h, [a]_{\sim_h}) \ | \ h \in X,\, h(a) = 1\}.\]
The function $q : S_\star \to X$ is defined by $q((h,[a]_{\sim_h}) := h$. For any $a \in S$, we define a function $s_a : \hat{a} \to S_\star$ by $s_a(h) := (h, [a]_{\sim_h})$. We now define the topology on $S_\star$ by taking the sets $\im(s_a)$ as a subbasis for the open sets, where $a$ ranges over $S$. 
\begin{lemma}\label{lem:sectcont}
Each function $s_a : \hat{a} \to S_\star$ is continuous and $q : S_\star \to X$ is an \'etale map.
\end{lemma}
\begin{proof}
Let $a, b \in S$ be arbitrary. We need to show that the set $s_a^{-1}(\im(s_b))$ is open in $X$. Notice that
\[s_a^{-1}(\im(s_b)) = \{ h \in X \ | \ a \sim_h b\} \cap \widehat{a}.\]
 Suppose $a \sim_h b$ and $h \in \widehat{a}$. Then also $h(b) = 1$. Pick $c, d \in S$ such that $h(c) = 0$, $h(d) = 1$ and $(a \wedge d) \vee c = (b \wedge d) \vee c$. Let $U_h := (\hat{c})^c \cap \hat{d} \cap \hat{a} \cap \hat{b}$. Then, for any $h' \in U_h$, we have $h'(a) = 1 = h'(b)$, $h'(c) = 0$ and $h'(d) = 1$, so that $a \sim_{h'} b$. So $h \in U_h \subseteq s_a^{-1}(\im(s_b))$.

To prove that $q$ is an \'etale map, let $e = (h, [a]_{\sim_h}) \in S_\star$. Then $q|_{\im(s_a)} : \im(s_a) \to \hat{a}$ has $s_a$ as its continuous inverse.
\end{proof}

\section{Proof of the main theorem}\label{s:proof1}
In this section, we will prove that the contravariant functor $(-)^\star : \Sh(\LPS) \to \DSL$ is full, faithful and essentially surjective. By a basic result from category theory (cf. for example \cite[Thm IV.4.1]{MacLane}) it then follows that $(-)^\star$ is part of a dual equivalence of categories, proving Theorem~\ref{th:main1}.

The proof that $(-)^\star$ is full and faithful is reasonably straightforward.
\begin{proposition}\label{prop:fullfaithful}
The contravariant functor $(-)^\star$ is full and faithful.
\end{proposition}
\begin{proof}
Let $E$ and $F$ be sheaves over local Priestley spaces $X$ and $Y$, respectively. We show that the assignment which sends a morphism $(f,\lambda) : (X,E) \to (Y,F)$ to $(f,\lambda)^\star : (Y,F)^\star \to (X,E)^\star$ is a bijection between the sets $\Hom_{\Sh(\LPS)}((X,E),(Y,F))$ and $\Hom_{\DSL}((Y,F)^\star,(X,E)^\star)$.

If $(f,\lambda)^\star = (g,\mu)^\star$ then in particular $\La(f) = \overline{(f,\lambda)^\star} = \overline{(g,\mu)^\star} = \La(g)$, using Lemma~\ref{lem:dslmorph}. Therefore, by classical Priestley duality, $f = g$. Moreover, if $U$ is a basic open set in $Y$, then $\lambda_U = \mu_U$, using the definition of $(f,\lambda)^\star = (g,\mu)^\star$. Since a natural transformation between sheaves is entirely determined by its action on basic opens (Lemma~\ref{lem:basicdet}), it follows that $\lambda = \mu$. So $(f,\lambda) = (g,\mu)$, proving that $(-)^\star$ is faithful.

If $h : (Y,F)^\star \to (X,E)^\star$ is a homomorphism of skew lattices, then $\overline{h}$ is a proper homomorphism, so by classical Priestley duality, there is a unique $f : X \to Y$ such that $\overline{h} = \La(f) = f^{-1}$. For $U$ a basic open, define $\lambda_U : F(U) \to E(f^{-1}(U))$ by sending $s \in F(U)$ to $h(s)$, which is indeed an element of $E(\overline{h}(s)) = E(f^{-1}(U))$. Now, if $U$ is an arbitrary open and $s \in F(U)$, we can write $U$ as a union of basic open sets $(U_i)_{i \in I}$. Then also $f^{-1}(U)$ is the union of the basic open sets $(f^{-1}(U_i))_{i \in I}$. It follows from the fact that $h$ is a homomorphism that $(h(s)|_{f^{-1}(U_i)})_{i \in I}$ is a compatible family, so there is a unique patch in $E(f^{-1}(U))$, which we define to be $\lambda_U(s)$. We leave it to the reader to check that $\lambda$ is a natural transformation and that $(f,\lambda)^\star = h$.
\end{proof}

The proof that $(-)^\star$ is essentially surjective is more involved, and will take up the rest of this section.
\\

Let $S$ be a left-handed strongly distributive skew lattice. By the construction from Section~\ref{s:lattice_to_space}, we have a sheaf $S_\star$ over the local Priestley space $X = \DSL(S,\2)$. Then $(S_\star)^\star$ is the skew lattice of local sections of $S_\star$ with compact open downward closed domains. We will show in the following three propositions that the map $\phi$, which sends $a \in S$ to $s_a \in (S_\star)^\star$ (cf. Lemma~\ref{lem:sectcont}), is an isomorphism of skew lattices. 
\begin{proposition}\label{prop:phihom}
The function $\phi : S \to (S_\star)^\star$ is a homomorphism of skew lattices.
\end{proposition}
\begin{proof}
It is clear that $\phi$ preserves $0$. Let $a, b \in S$. We need to show that $s_{a \vee b} = s_a \vee s_b$ and $s_{a \wedge b} = s_a \wedge s_b$. Note that in these equations, the operations $\vee$ and $\wedge$ on the right hand side are the operations defined in (\ref{eq:rqu}) and (\ref{eq:lqi}) of Section~\ref{s:space_to_lattice}, whereas the operations $\vee$ and $\wedge$ on the left hand side are the operations of the given left-handed strongly distributive skew lattice $S$.

Note that the domain of $s_{a \vee b}$ is $\widehat{a \vee b} = \hat{a} \cup \hat{b}$, which is also the domain of $s_a \vee s_b$. We now claim that $s_{a \vee b}(x) = s_b(x)$ for $x \in \hat{b}$ and $s_{a \vee b}(x) = s_a(x)$ for $x \in \hat{a} \setminus \hat{b}$, agreeing with the definition of $s_a \vee s_b$.
\begin{itemize}
\item Let $x \in \hat{b}$. For $d := b$ and $c := 0$, it is easy to show that $((a \vee b) \wedge d) \vee c = (b \wedge d) \vee c$, so $[a \vee b]_{\sim_x} = [b]_{\sim_x}$, by definition of $\sim_x$.
\item Let $x \in \hat{a} \setminus \hat{b}$. For $d := a$ and $c := b$, we then have $((a \vee b) \wedge d) \vee c = (a \wedge d) \vee c$, so that $[a \vee b]_{\sim_x} = [a]_{\sim_x}$.
\end{itemize}

Similarly, the domain of $s_{a \wedge b}$ is equal to the domain of $s_a \wedge s_b$, and if $x$ is an element of this domain, then we have $((a \wedge b) \wedge d) \vee c = (a \wedge d) \vee c$, for $d := a \wedge b$ and $c := 0$, proving that $[a \wedge b]_{\sim_x} = [a]_{\sim_x}$.
\end{proof}

To establish surjectivity of $\phi$, we will need the following lemma.
\begin{lemma}\label{lem:existence}
For each $n \in \mathbb{N}$, the following holds. 

If $s: U \to S_\star$ is a section on a compact open downward closed subset $U$ of $X$, and if $a_1, \dots, a_n$, $c_1, \dots, c_n$, $d_1,\dots, d_n$ are elements of $S$ such that
\begin{enumerate}
\item for each $i \in \{1,\dots,n\}$, $\hat{c_i} \subseteq \hat{d_i}$;
\item $U = \bigcup_{i=1}^n (\hat{d_i} \cap \hat{c_i}^c)$;
\item for each $i \in \{1, \dots, n\}$, $\hat{d_i} \cap \hat{c_i}^c \subseteq \hat{a_i}$, and $s|_{\hat{d_i} \cap \hat{c_i}^c} = s_{a_i}|_{\hat{d_i} \cap \hat{c_i}^c}$,
\end{enumerate}
then there exists an element $a \in S$ such that $s = s_a$.
\end{lemma}
\begin{proof}
By induction on $n \in \mathbb{N}$.
For $n = 0$, it follows from assumption (ii) that $U = \emptyset$, so $s$ is the empty function, and the (unique) element of $S$ such that $s = s_a$ is $a = 0$.

Now let $n \geq 1$, and assume the statement is true for $n-1$. Suppose that $s : U \to S_\star$, $a_1, \dots, a_{n}$, $c_1,\dots,c_{n}$, and $d_1, \dots, d_{n}$ satisfy the assumptions (i)--(iii).

Let $j \in \{1, \dots, n\}$ be arbitrary. We are going to apply the induction hypothesis to the function $s|_{\hat{c_j}} : \hat{c_j} \to S_\star$. For $i \in \{1,\dots,n\}$ with $i \neq j$, define $c_{i,j} := c_i \wedge c_j$ and $d_{i,j} := d_i \wedge c_j$. Note that
\[ \hat{d_{i,j}} \cap \hat{c_{i,j}}^c = (\hat{d_i} \cap \hat{c_i}^c) \cap \hat{c_j},\]
so that $\hat{c_j} = U \cap \hat{c_j} = \bigcup_{i \neq j} (\hat{d_{i,j}} \cap \hat{c_{i,j}}^c)$, and assumptions (i) and (iii) also clearly hold for the elements $a_i, c_{i,j}, d_{i,j}$, where $i$ ranges over $\{1,\dots,n\}\setminus\{j\}$. By the induction hypothesis, there exists an $f_j \in S$ such that $s|_{\hat{c_j}} = s_{f_j}$.

Since $j$ was arbitrary, we get that for each $j$ there exists an $f_j \in S$ such that $s|_{\hat{c_j}} = s_{f_j}$.
Now consider
\[a := \bigvee_{j=1}^n ((a_j \wedge d_j) \vee f_j).\]
We claim that $s = s_a$. Note first that
\[ \dom(s_a) = \hat{a} = \bigcup_{j=1}^n ((\hat{a_j} \cap \hat{d_j}) \cup \hat{f_j}) = \bigcup_{j=1}^n \hat{d_j} = U.\] 
Now let $x \in U$ be arbitrary, and let $j$ be the largest number in $\{1,\dots,n\}$ such that $x \in \hat{d_j}$. Then, using Proposition~\ref{prop:phihom} and the definition of $\vee$ in $(S_\star)^\star$, we see that
\[s_a(x) = \left\{ \begin{array}{cc} s_{f_j}(x) &\text{ if } x \in \hat{c_j}, \\
						  s_{a_j}(x) &\text{ if } x \not\in\hat{c_j}.
						  \end{array} \right.\]
If $x \in \hat{c_j}$, then $s_{f_j}(x) = s(x)$ by the choice of $f_j$, and if $x \not\in \hat{c_j}$, then $x \in \hat{d_j} \cap \hat{c_j}^c$, so $s_{a_j}(x) = s(x)$ by assumption (iii).
\end{proof}

The above lemma exactly enables us to prove surjectivity of $\phi$: it is now an application of compactness, as follows.
\begin{proposition}
The function $\phi : S \to (S_\star)^\star$ is surjective, and in particular it is a morphism of $\DSL$.
\end{proposition}
\begin{proof}
Let $s \in (S_\star)^\star$, so $s$ is a continuous section over a compact open downset $U$. 
For each $x \in U$, we have $s(x) \in (S_\star)_x = (S/\!\!\sim_x)^1$, so we can pick $a_x \in S$ such that $s(x) = [a_x]_{\sim_x}$, and define
\[ T_x := \|s = s_{a_x}\| = \{y \in U \ | \ s(y) = [a_x]_{\sim_y}\} = s^{-1}(\im(s_{a_x})) \cap U.\]
Now $T_x$ is open in $X$, because $s$ is continuous, $\im(s_{a_x})$ is open in $S_\star$, and $U$ is open in $X$. Since $x \in T_x$, there exist $c_x, d_x \in S$ such that $x \in \hat{d_x} \cap \hat{c_x}^c \subseteq T_x$, where we may assume without loss of generality that $\hat{c_x} \subseteq \hat{d_x}$. We now have
\[ U \subseteq \bigcup_{x \in U} (\hat{d_x} \cap \hat{c_x}^c) \subseteq \bigcup_{x \in U} T_x \subseteq U,\] 
so equality holds throughout.
Since $U$ is compact, there exist elements $x_1, \dots, x_n \in U$  such that $U = \bigcup_{i=1}^n (\hat{d_{x_i}} \cap \hat{c_{x_i}}^c)$. We will write $c_i$ and $d_i$ for $c_{x_i}$ and $d_{x_i}$, respectively. Note that, for each $i \in \{1,\dots,n\}$, we have $\hat{d_i} \cap \hat{c_i}^c \subseteq T_{x_i}$, so $s|_{\hat{d_i} \cap \hat{c_i^c}} = s_{a_{x_i}}|_{\hat{d_i} \cap \hat{c_i}^c}$. By Lemma~\ref{lem:existence}, we get $a \in S$ such that $s = s_a$, proving that $\phi$ is surjective.
For the in particular part, note that surjective homomorphisms are always proper.
\end{proof}

By a similar method, one may prove that $\phi$ is injective. Again, a lemma which is proved by induction is crucial.
\begin{lemma}\label{lem:uniqueness}
For each $n \in \mathbb{N}$, the following holds.

If $a, b$, $c_1, \dots, c_n$ and $d_1, \dots d_n$ are elements of $S$ such that:
\begin{enumerate}
\item for each $i \in \{1,\dots,n\}$, $\hat{c_i} \subseteq \hat{d_i} \subseteq \hat{a}$;
\item $\hat{a} = \bigcup_{i=1}^n (\hat{d_i} \cap \hat{c_i}^c) = \hat{b}$; 
\item for each $i \in \{1,\dots,n\}$, $(a \wedge d_i) \vee c_i = (b \wedge d_i) \vee c_i$,
\end{enumerate}
then $a = b$.
\end{lemma}
\begin{proof}
For $n = 0$, we get that $\hat{a} = \emptyset = \hat{b}$, so $a = 0 = b$.

Let $n \geq 1$, and suppose the statement is proved for $n-1$. Let $c_1, \dots, c_{n}$, $d_1,\dots d_{n}$ be elements of $S$ satisfying the assumptions. Then in particular $\hat{a} = \bigcup_{i=1}^{n} \hat{d_i} = \hat{\bigvee_{i=1}^{n} d_i}$, so that $[a]_\D = [\bigvee_{i=1}^{n} d_i]_\D$. Therefore,
\[ a = a \wedge \left(\bigvee_{i=1}^{n} d_i\right) = \bigvee_{i=1}^{n} (a \wedge d_i).\]
Similarly, since $\hat{b} = \bigcup_{i=1}^{n} \hat{d_i}$, we get that $b = \bigvee_{i=1}^{n} (b \wedge d_i)$.

Let $j \in \{1, \dots, n\}$ be arbitrary. For $i \neq j$, define $a_j := a \wedge c_{j}$, $b_j := b \wedge c_{j}$, $d_{i,j} := d_i \wedge c_{j}$, and $c_{i,j} := c_i \wedge c_{j}$. Note that $\hat{a_j} = \hat{b_j}$,
and also that for each $i \neq j$, we have $\hat{c_{i,j}} \subseteq \hat{d_{i,j}} \subseteq \hat{a_j}$. Moreover:
\begin{align*}
(a_j \wedge d_{i,j}) \vee c_{i,j} &= (a \wedge c_j \wedge d_i \wedge c_j) \vee (c_i \wedge c_j) &\text{(definitions of $a_j$, $d_{i,j}$ and $c_{i,j}$)}  \\
&= (a \wedge d_i \wedge c_{j}) \vee (c_i \wedge c_{j}) &\text{(left normality)} \\
&= ((a \wedge d_i) \vee c_i) \wedge c_{j} &\text{(strong distributivity)} \\
&= ((b \wedge d_i) \vee c_i) \wedge c_{j} &\text{(assumption)} \\
&= (b \wedge d_i \wedge c_{j}) \vee (c_i \wedge c_{j})  &\text{(as above)} \\
&= (b_j \wedge d_{i,j}) \vee c_{i,j}.
\intertext{By the induction hypothesis, we thus conclude that $a \wedge c_j = a_j = b_j = b \wedge c_j$. Now, to show $a \wedge d_j = b \wedge d_j$, we calculate:}
a \wedge d_j &= a \wedge (d_j \vee c_j) &\text{($\hat{c_j} \subseteq \hat{d_j}$)} \\
&= (a \wedge d_j) \vee (a \wedge c_j) &\text{(strong distributivity)} \\
&= (a \wedge d_j) \vee (b \wedge c_j) &\text{($a \wedge c_j = b \wedge c_j$)} \\
&= (a \wedge d_j) \vee c_j \vee (b \wedge c_j) &\text{($\hat{c_j} \subseteq \hat{b}$)} \\
&= (b \wedge d_j) \vee c_j \vee (b \wedge c_j) &\text{(assumption)} \\
&= (b \wedge d_j) \vee (b \wedge c_j) &\text{($\hat{c_j} \subseteq \hat{b}$)} \\
&= b \wedge d_j. &\text{(as above)}
\end{align*}
Now $a = \bigvee_{j=1}^{n} (a \wedge d_j) = \bigvee_{j=1}^{n} (b \wedge d_j) = b$, as required.
\end{proof}

\begin{proposition}
The function $\phi : S \to (S_\star)^\star$ is injective.
\end{proposition}
\begin{proof}
Let $a, b \in S$, and suppose that $s_a = s_b$. Then in particular $\hat{a} = \dom(s_a) = \dom(s_b) = \hat{b}$. For each $x \in \hat{a}$, we have $[a]_{\sim_x} = s_a(x) = s_b(x) = [b]_{\sim_x}$, so by definition of $\sim_x$, we may pick $c_x, d_x \in S$ such that $(a \wedge d_x) \vee c_x = (b \wedge d_x) \vee c_x$, and $x \in \hat{d_x} \cap \hat{c_x}^c$. We thus get that the collection $(\hat{d_x} \cap \hat{c_x}^c)_{x \in \hat{a}}$ is an open cover of $\hat{a}$. Since $\hat{a}$ is compact, we can pick a finite subcover, indexed by $x_1, \dots, x_n \in \hat{a}$. We will write $c_i$ and $d_i$ for $c_{x_i}$ and $d_{x_i}$, respectively.
Without loss of generality, we may assume that $\hat{c_{i}} \subseteq \hat{d_{i}} \subseteq \hat{a}$ for each $i$, by replacing $c_{i}$ by $c_{i} \wedge d_{i} \wedge a$ and $d_{i}$ by $d_{i} \wedge a$, and checking that the new $c_{i}$ and $d_{i}$ still satisfy the same properties. 
Now it follows from Lemma~\ref{lem:uniqueness} that $a = b$.
\end{proof}

We have thus established that $\phi : S \to (S_\star)^\star$ is an isomorphism in $\DSL$, so:
\begin{proposition}\label{prop:esssurj}
The contravariant functor $(-)^\star : \Sh(\LPS) \to \DSL$ is essentially surjective.
\end{proposition}

It now follows from Propositions \ref{prop:fullfaithful} and \ref{prop:esssurj} that $(-)^\star$ is part of a dual equivalence. This concludes the proof of our main theorem, Theorem~\ref{th:main1}.


\section{Concluding remarks}\label{s:conc}
It is a central fact in logic  that every distributive lattice embeds in a unique Boolean algebra. This fact is at the base of the relationship between intuitionistic and Boolean logic. It would be interesting to seek a non-commutative counterpart of this result. Since the classical result is most transparently understood via duality, it is likely that our duality would prove useful. Furthermore, a non-commutative Heyting algebra is a notion still needing to be properly defined. In a recent paper \cite{Gehrke} on Esakia's work, Gehrke showed that Heyting algebras may be understood as those distributive lattices for which the embedding in their Booleanization has a right adjoint. This could provide a natural starting point for the exploration of skew Heyting algebras.

A different natural non-commutative generalization of distributive lattices is a class of inverse semigroups whose idempotents form a distributive lattice. Recently, Stone duality has been generalized to this setting \cite{Law,Law1,LL}. The most recent work in this direction \cite{LL} generalizes Stone's duality between distributive lattices and spectral spaces \cite{Sto1937} to the context of inverse semigroups. However, to the best of our knowledge, 
Priestley's duality \cite{P1} for distributive lattices has not yet been generalized to inverse semigroups.
The results in this paper might also be fruitfully applied to obtain such a duality for a class of inverse semigroups. We leave this as an interesting direction for future work.

\section*{Acknowledgements}

The work of AB was supported by ARRS grant P1-0294. The work of KCV was supported by ARRS grant P1-0222. The work of MG was partially supported by ANR 2010 BLAN 0202 02 FREC. The PhD research project of SvG has been made possible by NWO grant 617.023.815 of the Netherlands Organization for Scientific Research. The work of GK was supported by ARRS grant P1-0288. We would also like to thank the referee for comments leading to improvements in the introduction of the paper.

\end{document}